\documentclass[12pt]{amsart}
\usepackage{amsmath}
\usepackage{amssymb}
\usepackage{amstext}
\usepackage{amscd}
\usepackage[matrix,arrow,ps]{xy}

\newtheorem{teo}{Theorem}[section]
\newtheorem{prop}[teo]{Proposition}
\newtheorem{lem}[teo]{Lemma}

\newtheorem{exe}[teo]{Example}
\newtheorem{defini}[teo]{Definition}

\newcommand{\SL}{{\rm SL}}

\newcommand{\Res}{{\rm Res}}
\newcommand{\MT}{{\rm MT}}

\newcommand{\CC}{{\mathbb C}}
\newcommand{\RR}{{\mathbb R}}
\newcommand{\ZZ}{{\mathbb Z}}
\newcommand{\QQ}{{\mathbb Q}}

\newcommand{\HH}{{\mathbb H}}

\newcommand{\GG}{{\mathbb G}}
\newcommand{\SSS}{{\mathbb S}}
\newcommand{\AAA}{{\mathbb A}}

\newcommand{\lto}{\longrightarrow}

\newcommand{\cH}{{\mathcal H}}
\newcommand{\cK}{{\mathcal K}}

\newcommand{\cX}{{\mathcal X}}

\newcommand{\ol}{\overline}

\newcommand{\wt}{\widetilde}

\title{Algebraic flows on Shimura varieties.}
\author{Emmanuel Ullmo, Andrei Yafaev}
\begin{document}

\maketitle

\tableofcontents

\section{Introduction.}
In the paper \cite{UY_AV} we formulated certain conjectures about algebraic
flows on abelian varieties and proved certain cases of these conjectures.
The purpose of this paper is two-fold. We first prove the `logarithmic Ax-Lindemann theorem' 
(see details below). We then prove a result analogous to one of the main results of \cite{UY_AV}
in the hyperbolic (Shimura) case about the topological closure of the images of 
totally geodesic subvarieties of the symmetric spaces uniformising Shimura varieties.

Let $(G,X)$ be a  Shimura datum and $X^+$ be a connected component of $X$.
Recall from \cite{Ul}, section 2.1 that a realisation $\cX$ of $X^+$  is a complex quasi-projective variety $\cX$ with a transitive holomorphic action of $G(\RR)^+$
such that for any $x_0 \in \cX$, the orbit map  $\psi_{x_0} \colon G(\RR)^+ \lto \cX$ mapping $g$ to $g x_0$ is
semi-algebraic.
 There is a natural notion of a morphism of realisations.
By \cite{Ul}, lemma 2.1, any realisation of $X^+$ has a canonical
semi-algebraic structure and any morphism of realisations is semi-algebraic. 

In what follows we fix a realisation $\cX$ of $X^+$ and by a slight abuse of language still call this 
realisation $X^+$. 
It is an immediate consequence of Lemma 2.1 of \cite{Ul} that all the conjectures and statements that follow are
independent of the chosen realisation.

In view of the lemma B1 of \cite{KUY}, we may define an algebraic subset $Y$
of $X^+$ to be a closed analytic, semi-algebraic subset of $X^+$.
Given an irreducible analytic subset $\Theta \subset X^+$, we define the Zariski closure of 
$\Theta$ to be the analytic component containing $\Theta$ of the smallest algebraic subset of $X^+$
containing $\Theta$.

We can now state some results and conjectures.

The classical formulation of the hyperbolic Ax-Lindemann theorem is as follows:

\begin{teo}[Hyperbolic Ax-Lindemann theorem]
Let $S$ be a Shimura variety and $\pi \colon X^+ \lto S$
be the uniformisation map.
Let $Z$ be an algebraic subvariety of $S$ and $Y$ a 
maximal algebraic subvariety of $\pi^{-1}(Z)$.
Then $\pi(Y)$ is a weakly special subvariety of $S$.
\end{teo}

We will see (see proposition \ref{equivalence}) that this is equivalent to:

\begin{teo}[Hyperbolic Ax-Lindemann theorem, version 2.]
Let $Z$ be any irreducible algebraic subvariety of $X^+$ then the Zariski closure of $\pi(Z)$
is weakly special.
\end{teo}

The hyperbolic Ax-Lindemann conjecture has been proven in full generality in \cite{KUY}.

In the second section we define a notion of a weakly special subvariety of $X^+$.
This is  a complex analytic subset $\Theta$ of $X^+$ such that there exists a 
semi-simple algebraic subgroup $F$ of $G(\RR)^+$  and a point $x \in X^+$ 
satisfying certain conditions such that
$\Theta = F \cdot x$.

In Section 3 of this paper
we prove a `logarithmic'  Ax-Lindemann theorem
(a question asked by D. Bertrand). 

\begin{teo}[Logarithmic Ax-Lindemann]
Let $\pi \colon X^+ \lto S$ be the uniformisation map.
Let $Y$ be an algebraic subvariety of $S$ and let $Y'$ be
an analytic component of $\pi^{-1}(Y)$. The Zariski closure of
$Y'$ is a weakly special subvariety.
\end{teo}

In \cite{UY_AV}, we formulated two conjectures on algebraic flows on abelian varieties and proved 
partial results towards these conjectures.
An attempt to formulate conjectures of this type in the context of Shimura varieties displays new phenomena that we intend to investigate in 
the future.
We however prove a result which may be seen as a generalisation in the context of Shimura varieties 
of one of the main results of \cite{UY_AV}. To state our result we need to introduce a few notations.

Consider an algebraic subset $\Theta$ of $X^+$. 
In general, instead of (as in the hyperbolic Ax-Lindemann case)  being interested in the Zariski closure
of $\pi(\Theta)$, we look at the usual topological closure $\overline{\pi(\Theta)}$.
We define a notion of \emph{real weakly special} subvariety roughly as the
image of $H(\RR)\cdot x$ where $H$ is a semisimple subgroup of $G$
satisfying certain conditions and $x$ is a point of $X^+$.
Let $K_x$ be the stabiliser of $x$ in $G(\RR)^+$.
In the case where $H(\RR)^+ \cap K_x$ is a maximal compact subgroup,
 a real weakly special subvariety of $S$ is 
a \emph{real} totally geodesic subvariety of $S$.
Notice that in this case the homogeneous space $H(\RR)^+/H(\RR)^+ \cap K_x$ is a 
real symmetric space. 
In the case where $x$ viewed as a morphism from $\SSS$ to $G_{\RR}$
factors through $H_{\RR}$, the corresponding real weakly special 
subvariety has Hermitian structure and in fact is a weakly special subvariety in
the usual sense.
We also note that given a real weakly special subvariety $Z$ of $S$, there is a canonical
probability measure $\mu_Z$ attached to $Z$ which is the pushforward of the Haar
measure on $H(\RR)^+$, suitably normalised to make it a probability measure.

In this paper we prove the following  theorem.

\begin{teo}\label{t1}

Let $\Theta$ be a complex totally geodesic subvariety of $X$. 
Then the components of the topological closure $\overline{\pi(\Theta)}$ are 
real weakly special subvarieties.
\end{teo}

Recall that a complex totally geodesic subvariety of $X^+$ is
of the form $F \cdot x$ where $F$ is a semisimple real Lie group
subject to certain conditions and $x$ is a point of $X$ such that $F \cap K_x$ is
a maximal compact subgroup of $F$.


In certain cases, for example when the centraliser of $F$ in $G(\RR)$ is trivial,
we are able to show that $\overline{\pi(\Theta)}$ is actually a 
(complex) weakly special subvariety.
This condition is satisfied in many cases. For example in the 
case of $\SL_2(\RR)$ diagonally embedded into a product of copies of $\SL_2(\RR)$.
In particular this answers in the affirmative the question of Jonathan Pila which was the following.
Consider the subset $Z$ of $\HH \times \HH$ which is 
$$
Z = \{ (\tau, g\tau) : \tau \in \HH \}
$$
where $g \in \SL_2(\RR)\backslash \SL_2(\QQ)$.
Is the image of $Z$ dense in $\CC \times \CC$?

The proof of Theorem \ref{t1} relies on the results of Ratner (see \cite{Rat})
on closure of unioptent one parameter subgroups in homogeneous spaces.

\section*{Acknowledgements.}

We thank Jonathan Pila for discussions around the topic of the second part of this paper.
We also thank Daniel Bertrand who raised the question of Logarithmic Ax-Lindemann theorem.
We are very grateful to Ngaiming Mok for many stimulations discussions.

\section{Weakly special subvarieties and monodromy.}

\subsection{Monodromy.}
Let $(G,X)$ be a Shimura datum. Recall that $G$ is a reductive group over $\QQ$ such that
$G^{ad}$ has no $\QQ$-simple factor whose real points are compact and $X$ is a $G(\RR)$-conjugacy class of a morphism $x \colon \SSS \lto G_{\RR}$ where $\SSS=\Res_{\CC/\RR}\GG_{m,\CC}$. The morphism $x$ is required to satify Deligne's conditions which imply that components of $X$ are Hermitian symmetric domains.
There is a natural notion of morphisms of Shimura data.
We fix a connected component $X^+$ of $X$ and we let $\Gamma = G(\QQ)_+ \cap K$
where $G(\QQ)_+$ is the stabiliser of $X^+$ in $G(\QQ)$.
Let $S$ be $\Gamma \backslash X^+$ and $\pi \colon X^+ \lto S$
be the natural morphism. 

To $(G,X)$, one associates the adjoint Shimura datum $(G^{ad}, X^{ad})$ with a 
natural morphism $(G,X) \lto (G^{ad}, X^{ad})$ induced by the natural map $G \lto G^{ad}$.
Notice that the this map identifies $X^+$ with a connected component of $X^{ad}$.
We have the following description of weakly special (or totally geodesic) subvarieties (see Moonen \cite{MoMo}):
\begin{teo}
A subvariety $Z$ of $S$ is totally geodesic if and only if there exists a sub-datum $(M,X_M)$
of $(G,X)$ and a product decomposition
$$
(M^{ad}, X_M^{ad}) = (M_1,X_1) \times (M_2, X_2)
$$
and a point $y_2$ of $X_2$ such that $Z = \pi(X^+_1 \times y_2)$ for a component
$X_1^+$ of $X_1$.
\end{teo}

Note that $X_M^{ad,+}=X_1^+ \times X_2^+$ (with a suitable choice of connected components) is a subspace of $X^+$.



We can without loss of any generality assume the group $\Gamma$ to be neat, i.e.
the stabiliser of each point of $X^+$ in $\Gamma$ to be trivial (replacing $\Gamma$ by a subgroup of finite index changes nothing to the property of a subvareity to be weakly special).
Fix a point $x$ of the smooth locus $Z^{sm}$ and $\wt{x} \in \pi^{-1}(x)\cap Z^{sm}$.
This gives rise to the monodromy representation
$$
\rho^m \colon \pi_1(Z^{sm}, x) \lto \Gamma
$$
whose image we denote by $\Gamma^m$.
By Theorem 1.4 (due to Deligne and Andr\'e) of \cite{MoMo}, we have $\Gamma^m \subset M^{der}(\QQ)\cap \Gamma$.

This can all be summarised in the following theorem.

\begin{teo} \label{monodromy}
Let $(G,X)$ be Shimura datum, $K$ a compact open subgroup
of $G(\AAA_f)$ and $\Gamma:= G(\QQ)_+ \cap K$ (assumed neat).

Let $S = \Gamma \backslash X^+$ and $Z$ an irreducible subvariety of $S$.
Let $M$ be the generic Mumford-Tate group on $Z$ and $X_M$ the $M(\RR)$-conjugacy class of $x$. 

Let $\Gamma^m \subset M^{der}(\QQ)\cap \Gamma$ be the
monodromy group attached to $Z$ as described above.

Let $(M^{ad},X_M^{ad}) = (M_1,X_1)\times (M_2,X_2)$  as in Theorem 4.3 of \cite{MoMo}.
In particular $M_1$ is the image of the neutral component of the Zariski closure of $\Gamma^m$ in $M^{ad}$.

Let $K_M^{ad} = K_1 \times K_2$ be a compact open subgroup containing the image of $K_M = M(\AAA_f)\cap K$
in $M^{ad}(\AAA_f)$ (here $K_i$s are compact open subgroups of $M_i(\AAA_f)$).
We let $X_i$ be the $M_i(\RR)$-conjugacy classes of $x$.

Let $S_M\subset S$ be a connected component of the image of $Sh_{M(\AAA_f)\cap K}(M,X_M)$
in $S$ containing $Z$.

Let $S_i$ ($i=1,2$) be appropriate components of $Sh_{K_i}(M_i,X_i)$ and
$S_M \lto S_ 1\times S_2$ be the natural map.

The image of $Z$ in $S_1 \times S_2$ is of the form
$Z_1 \times \{ z \}$ (see Theorem 4.3 of \cite{MoMo}) where $Z_1$ is a subvariety
of $S_1$ whose monodromy is Zariski dense in $M_1$and $z$ is a point of $S_2$.
\end{teo}

\subsection{Weakly special subvarieties of $X^+$.}

In this section we give a precise description of totally geodesic 
(weakly special) subvarieties of $X^+$.

Let $(G,X)$ be a Shimura datum and $X^+$ a connected component of $X$.
For the purposes of this section, we can without loss of generality assume that $G$
is a semi-simple group of adjoint type. This is because there is a natural identification between connected components of $X^+$ and a connected component of $X^{ad}$.
We will now describe totally geodesic subvarieties of $X^+$ (that we will naturally call weakly special).

The group $G$ has no $\QQ$-simple factors whose real points are compact and there is 
 a morphism $x_0:\SSS \longrightarrow G_{\RR}$ satisfying  the following Deligne's conditions
such that $X^+=G(\RR)^+.x_0$.

(D1) The adjoint representation $\mbox{Lie}(G_{\RR})$ is of type 
$\{(-1,1), (0,0),(1,-1\}$. In particular $x(\GG_{m,\RR})$ is trivial.

(D2) The involution $x(\sqrt{-1})$ of $G_{\RR}$ is a Cartan involution.

This is a consequence of \cite{De} 1.1.17.




We have the following:

\begin{prop} \label{factoring}
Let $Z$ be a totally geodesic complex subvariety of $X^+$. There exists 
a semi-simple real algebraic subgroup $F$ of $G_{\RR}$ without compact factors  and some $x\in X$
such that $x$ factors through $FZ_{G}(F)^0$ such that $Z=F(\RR)^+.x$.
Conversely, let $F$ be  a semi-simple real algebraic subgroup of $G_{\RR}$ without compact factors  and let $x\in X$
such that $x$ factors through $FZ_{G}(F)^0$. Then $F(\RR)^+.x$ is a totally geodesic subvariety of $X^+$.
\end{prop}
\begin{proof}
 let $F$ be  a semi-simple real algebraic subgroup of $G_{\RR}$ without compact factors  and let $x\in X$
such that $x$ factors through $H:=FZ_{G}(F)^0$. Then 
$$
Z_{G}(H(\RR))\subset Z_{G}(x(\sqrt{-1})). 
$$
As $Z_{G}(x(\sqrt{-1}))$ is a compact subgroup of $G(\RR)$ 
so is $Z_{G}(H(\RR))$. By using \cite{Ul2} lemma 3.13 we see that
$H$ is reductive. 

Then the proof of \cite{Ul2} lemma 3.3 shows that $X_H:=H(\RR)^+.x$
is an hermitian symmetric subspace of $X^+$. We give the arguments
to be as self contained as possible. 

As ${\rm Lie }(H_{\RR})$ is a sub vector space  of ${\rm Lie }(G_{\RR})$ the Hodge weights of 
${\rm Lie }(H_{\RR})$ are $\{(-1,1), (0,0), (1,-1)\}$. Then using Deligne \cite{De} 1.1.17
we just need to prove that $x(\sqrt{-1})$ induces a Cartan involution of $H^{ad}$. 
As the square of $x(\sqrt{-1})$ is in the centre of $H(\RR)$, by Deligne \cite{De} 1.1.15, it's enough to check that
$H_{\RR}$ admits a faithful real $x(\sqrt{-1})$-polarizable representation $(V,\rho)$.
We may take $V={\rm Lie G_{\RR}}$ for the adjoint representation and the $x(\sqrt{-1})$-polarization
induced from the Killing form $B(X,Y)$.

Then $H_{\RR}$ is the almost direct product $H_{\RR}\simeq F F_1^{nc}F_1^c$
where $F_1$ is either trivial or  semi-simple without compact factors and $F_1^c$ is 
reductive  with $F_1^c(\RR)$ compact. If $F_1^{nc}$ is trivial $X_F^+=X_H^+$ is hermitian symmetric.
If $F_1^{nc}$ is not trivial, we have a decomposition $X_H^+=X_{F}^+\times X_{F_1^{nc}}^+$
is a product of hermitian subspaces and we have the natural identification of 
$X_F^+$ with $X_F^+\times \{x_1\}$ where $x_1$ is the projection of $x$ on $X_{F_1^{nc}}^+$.
In any case $X_F^+$ is hermitian symmetric and totally geodesic in $X^+$.

Conversely a totally geodesic subvariety of $X^+$ is of the form $X_F^+=F(\RR)^+.x$
for a semi-simple subgroup $F_{\RR}$ of $G_{\RR}$ without compact factors.  
Let $T_x(X_F^+)\subset T_x(X^+)$  be the tangent space of $X_F^+$ at $x$.
Let $U^1\subset \SSS$ be the unit circle. 
The complex structure on $T_x(X^+)$ is given by the adjoint action of $x(U^1)$.
If $X_F$ is a complex subvariety, then  $T_x(X_F^+)$ is stable by 
 $x(U^1)$. Using Cartan decomposition we see that  $x(U^1)=x(\SSS)$ normalizes $F$.
 Let $F_1=x(\SSS) F$, then $F_1$ is reductive and is contained in $FZ_{G}(F)^0$.
 It follows that $x$ factors through  $FZ_{G}(F)^0$.

\end{proof}

\begin{defini}
An algebraic group $H$ over $\QQ$ is said to be of type $\cH$ if its 
radical is unipotent and if $H/R_u(H)$ is an almost direct product 
of $\QQ$ simple factors $H_i$ with $H_i(\RR)$ non-compact.
Furthermore we assume that at least one of those factors not to be trivial.
\end{defini}

Let $H \subset G$ be a subgroup of type $H$ and let us assume that $G$ is of adjoint type.
We will now explain how to attach a hermitian symmetric space $X_H$ to
a group of type $\cH$ and explain that $X_H$ is independent of the 
choice of a Levi subgroup in $H$.

The domain $X^+$ is the set of maximal compact subgroups of $G(\RR)^+$.
Let $x\in X^+$,  we denote by $K_x$ the associated maximal compact subgroup
of $G(\RR)^+$.  Let $H$ be a subgroup of type $\cH$ and let $L$ be a Levi subgroup
of $H$. We have a Levi decomposition $H=R_u(H).L$. 
Assume that $K_x\cap L(\RR)^+$ is a maximal compact subgroup of $L(\RR)^+$.
Then $X_L^+=L(\RR)^+.x\subset X^+$ is the symmetric space associated to $L$
and is independent of the choice of $x\in X^+$ such that $K_x\cap L(\RR)^+$
is a maximal compact subgroup of $L(\RR)^+$. Let $X_H^+:=R_u(H)X_L(\RR)^+$,
then $X_H^+$ is independent of the chosen Levi decomposition of $H$.
This can be seen as follows. The Levi subgroups of $H$ are conjugate by an
element of $R_u(H)$. Let $L'$ be a Levi of $H$ and $w\in R_u(H)$ such that
$L'=wLw^{-1}$. Let $x'=w.x$. Then $K_{x'}$ is a maximal compact subgroup of $G(\RR)^+$
such that $K_{x'}\cap L'(\RR)^+$ is a maximal compact subgroup of $L'(\RR)^+$ and
$$
R_u(H).X_{L'}^+=R_u(H).L'(\RR)^+.x'=R_u(H).wL(\RR)^+.x=R_u(H).X_L^+.
$$
This shows that the space $X_H^+$ is independent of the choice of the Levi.

\begin{defini}
A real weakly special subvariety of $S$ is  a real analytic subset of $S$
of the form
$$
Z=\Gamma\cap H(\RR)^+\backslash H(\RR)^+.x
$$
where $H$ is an algebraic subgroup of $G$ of type $\cH$ and $x\in X^+$.

In the case where $K_x\cap L(\RR)^+$ is a maximal compact subgroup of $L(\RR)^+$
for some Levi subgroup of $H$,  $H(\RR)^+/K_x \cap H(\RR)^+$ is a real symmetric space.

\end{defini}

We have the following proposition.

\begin{prop}
Let $Z$ be a real weakly special subvariety of $S$.
Then the Zariski closure $Z^{Zar}$ of $Z$ is
weakly special.
\end{prop}
\begin{proof}
By definition, $Z$ is of the form $Z = H(\RR)^+ \cdot x$ where
$H$ is a group of type $\cH$.
Let $S_M$ be as in Theorem \ref{monodromy} the smallest special subvariety
containing $Z^{Zar}$.

Let $S_1 \times S_2$ be the product of Shimura varieties as in Theorem \ref{monodromy}
such that the image of $Z^{Zar}$ in $S_1 \times S_2$ is of the form $Z_1 \times \{ z \}$
where $Z_1$ is a subvariety of $S_1$ whose monodromy $\Gamma^m_1$ is Zariski dense in $M_1$ and
$z$ is a Hodge generic point of $S_2$.

To prove that $Z^{Zar}$ is weakly special, it is enough to show that $Z_1 = S_1$.
In what follows, we replace $S$ by $S_1$ and $Z$ by $Z_1$.


For any $ q \in H(\QQ)^+$, we have
that $Z \subset T_qZ$, therefore
 $$Z \subset Z^{Zar}\cap T_q(Z^{Zar}).$$
Since $Z^{Zar}\cap T_q(Z^{Zar})$ is algebraic, we have
$$
Z^{Zar} \subset Z^{Zar}\cap T_q(Z^{Zar})
$$
and therefore, for each $q \in H(\QQ)$ we have
$$
Z^{Zar}\subset T_q(Z^{Zar}).
$$
Let $T$ be a non-trivial subtorus of $H$.
We define the Nori constant $C(Z^{Zar})$ of $Z^{Zar}$ as in \cite{UY1}, section 4.
Let $p > C(Z^{Zar})$ and $q \in T(\QQ)$ given by Lemma 6.1 of \cite{UY1}.
Then $T_q(Z^{Zar})$ is irreducible and the orbits of $T_q + T_{q^{-1}}$ are dense in $S$.
This implies that
$Z^{Zar} = S$ as required.
\end{proof}

\section{Logarithmic Ax-Lindemann.}

Let $S = \Gamma \backslash X^+$ as before and consider a realisation $X^+ \subset \CC^n$
(in the sense of \cite{Ul2}).  In particular $X^+$ is a semi-algebraic set and the action of $G(\RR)^+$ is semi-algebraic.

Let $\wt{Y}$ be a complex analytic subset of $X^+$. Then the Zariski closure $\ol{\wt{Y}}^{Zar}$
in $\CC^n$ is an algebraic subset of $\CC^n$ and $\ol{\wt{Y}}^{Zar} \cap X^+$ has finitely many
analytic components.
By slight abuse of notation, we refer to $\ol{\wt{Y}}^{Zar} \cap X^+$ as \emph{Zariski closure of $\wt{Y}$}.
These components are algebraic in the sense of the definition given in the Appendix B of \cite{KUY}.

\begin{teo}  [Logarithmic Ax-Lindemann]
Let $\pi \colon X^+ \lto S$ be the uniformisation map.
Let $Y$ be an algebraic subvariety of $S$ and let $Y'$ be
an analytic component of $\pi^{-1}(Y)$. The Zariski closure of
$Y'$ is a weakly special subvariety.
\end{teo}
\begin{proof}

 Let $\wt{Y}$ be an analytic component of $Y'$.
As in the previous section we can replace $S$ by $S_1$ and $Y$ by $Y_1$ given by the 
Proposition \ref{monodromy}. In doing this we attach the monodromy to a point
$y \in Y^{sm}$ and $\wt{y} \in Y'$.
Let $\Gamma_Y$ be the monodromy group attached to $Y$.
Notice that $\Gamma_Y$ is the stabiliser of $Y'$ in $\Gamma$.
Then, with our assumptions, $\Gamma_Y$ is Zariski dense in $G$.

Let $\alpha \in \Gamma_Y$. 
We have 
$$
\alpha Y' = Y'
$$
Therefore,
$$
{(\alpha Y')}^{Zar}  = {Y'}^{Zar}
$$
We also have 
$$
\alpha {Y'}^{Zar} \supset \alpha Y'
$$
and since $\alpha {Y'}^{Zar} $ is algebraic, we have
$$
\alpha {Y'}^{Zar}  \supset {\alpha Y'}^{Zar}
$$
The same argument with $\alpha^{-1}$ instead of $\alpha$
shows that the reverse inclusion holds and therefore
$$
\alpha {Y'}^{Zar}  = {\alpha Y'}^{Zar} = {Y'}^{Zar}
$$
It follows that ${Y'}^{Zar}$ is stabilised by $\Gamma_Y$.

Consider the stabiliser $G_Y$ of ${Y'}^{Zar}$ in $G(\RR)$.
Since ${Y'}^{Zar}$ is semi-algebraic and the action of $G(\RR)^+$ on
$X^+$ is semi-algebraic, $G_Y$ is semi-algebraic.
Furthermore, $G_Y$ is analytically closed and hence is a real algebraic group.
Since $G_Y$ contains $\Gamma_Y$ which is Zariski dense in $G_{\RR}$, we see that
$G_Y = G(\RR)^+$.
 It follows that ${Y'}^{Zar} = S$ as required.
\end{proof}

\section{Facts from ergodic theory: Ratner's theory.}

In this section  we recall some results from ergodic theory of homogeneous 
varieties to be used in the next section.
The contents of this section are mainly taken from  Section 3 of \cite{CU1}.
We present results in the way they are presented in \cite{Ul}.

Let $G$ be a semi-simple algebraic group over $\QQ$. 
We assume that $G$ has no $\QQ$-simple simple factors that
are anisotropic over $\RR$. This condition is satisfied by all groups 
defining Shimura data.

Let
$\Gamma$ be an arithmetic lattice in $G(\RR)^+$ and let
$\Omega = \Gamma \backslash G(\RR)^+$.



We have already defined a subgroup $H\subset G$ of type $\cH$,
we now define a group of type $\cK$.

\begin{defini} \label{typeK}
Let $F \subset G(\RR)$ be a closed connected Lie subgroup.
We say that $F$ is of type $\cK$ if

\begin{enumerate}
\item $F \cap \Gamma$ is a lattice in $F$.
In particular $F \cap \Gamma \backslash F$ is closed in $\Gamma \backslash G(\RR)^+$.
We denote by $\mu_F$ the $F$-invariant normalised measure on $\Gamma \backslash G(\RR)^+$.
\item
The subgroup $L(F)$ generated by one-parameter unipotent subgroups
of $F$  acts ergodically on $F \cap \Gamma \backslash F$ with respect to $\mu_F$.
\end{enumerate} 

For the purposes of this section, we in addition assume $F$ to be semisimple.
\end{defini}

The relation between types $\cK$ and $\cH$ is as follows (see \cite{CU}, lemme 3.1 and 3.2):
\begin{lem} \label{Equivalence}
\begin{enumerate}
\item If $H$ is of type $\cH$, then $H(\RR)^+$ is of type $\cK$.
\item It $F$ is a closed Lie subgroup of $G(\RR)^+$ of type $\cK$, then there exists
a $\QQ$ subgroup $F_{\QQ}$ of $G$ of type $\cH$ such that
$F = F(\RR)^+$.
\end{enumerate}
\end{lem}

For a subset $E$ of $G(\RR)$, we define the Mumford-Tate group $MT(E)$ of $E$ as
the smallest $\QQ$-subgroup of $G$ whose $\RR$-points contain $E$.
If $F$ is a Lie subgroup of $G(\RR)^+$ of type $\cK$ , then 
by (2) of the above lemma, $MT(F) = F_{\QQ}$ and it is of type $\cH$.

We will make use of the following lemma, which is Lemma 2.4 of \cite{Ul}.

\begin{lem} \label{Almost_simple}
Let $H$ be a $\QQ$-algebraic subgroup of $G$ with $H^0$ almost simple.
Let $L$ be an almost simple factor of $H^0_{\RR}$. Then 
$$
MT(L)=H^0
$$
\end{lem}

Let $\Omega = \Gamma\backslash G(\RR)^+$.
Note that $\Omega$ carries a natural probability measure, the pushforward of
the Haar measure on $G(\RR)^+$, normalised to be a probability measure (the volume of $\Omega$ is finite).
For each $F$ of type $\cK$, there is a natural probability measure $\mu_F$ attached to $F$.


The following theorem is a consequence of results of Ratner. 
\begin{teo} \label{Ratnerthm}

Let $F=F(\RR)^+$ be a subgroup of $G(\RR)^+$ 
be a semi-simple group without compact factors.

Let $H$ be $\MT(F)$. The closure of $\Gamma \backslash \Gamma F$ in
$\Omega$ is 
$$\Gamma \backslash \Gamma H(\RR)^+ = \Gamma\cap H(\RR)^+ \backslash H(\RR)^+$$
\end{teo}
\begin{proof}
By a result of Cartan(\cite{PlaRa}, Proposition 7.6) the group $F$ is generated by 
its one-parameter unipotent subgroups.

A result of Ratner (see \cite{Rat}, Theorem 3), implies that the closure of 
$\Gamma\backslash\Gamma F$ in $\Omega$ is homogeneous i.e.
there exists  a Lie group $H$ of type $\cK$ such that 
$$
\overline{\Gamma\backslash\Gamma F} = \Gamma\backslash\Gamma H
$$ 
By Lemme 2.1(c) of \cite{CU}, there exists a $\QQ$-algebraic subgroup $H_{\QQ}\subset G$
such that
$$
H(\RR)^+ = H
$$
Since $F \subset H$, we have that $MT(F)\subset H$.
On the other hand, by Lemme 2.2 of \cite{CU} (due to Shah), the radical of $MT(F)$
is unipotent which implies that $MT(F)$ is of type $\cH$.
It follows that $H_{\QQ} = MT(F)$ which finishes the proof.
\end{proof}

\section{Algebraic flows on Shimura varieties.}

\subsection{Reformulation of the hyperbolic Ax-Lindemann theorem.}
Let $(G,X)$ be a Shimura datum. Let $K$ be a compact open subgroup of $G(\AAA_f)$,
$\Gamma=G(\QQ)_+\cap G(\AAA_f)$ and 
$S=\Gamma\backslash X^+$. Let $\pi \colon X^+\rightarrow S$ be the uniformizing map.
Without loss of any generality, in this section we assume that the group $G$ is semisimple of adjoint type.

We first give a reformulation of the hyperbolic Ax-Lindemann conjecture in terms of algebraic flows.
\begin{prop} \label{equivalence}
The hyperbolic Ax-Lindemann conjecture is equivalent to the following statement.
Let $Z$ be any irreducible algebraic subvariety of $X^+$ then the Zariski closure of $\pi(Z)$
is weakly special.
\end{prop}
\begin{proof}
Let us assume that the hyperbolic Ax-Lindemann conjecture holds true. Let $A$ be an irreducible
algebraic subvariety of $X^+$ and $V$ be the Zariski closure of $\pi(A)$. Let $A'$ be a maximal irreducible
algebraic subvariety of $\pi^{-1}(V)$ containing $A$. By the hyperbolic Ax-Lindemann conjecture
$\pi(A')$ is a weakly special subvariety of $V$. As $A\subset \pi(A')\subset V$ and as $\pi(A')$ is  irreducible algebraic
we have $V=\pi(A')$. Therefore $V$ is weakly special.

Let us assume that the statement of the proposition holds true. 
Let $V$ be an irreducible algebraic subvariety of $S$. Let $Y$ be a maximal irreducible
algebraic subvariety of $\pi^{-1}(V)$. Then the Zariski closure $V'$ of $\pi(Y)$ is weakly special.
Moreover $V'\subset V$. Let $W$ be an analytic component of $\pi^{-1}(V')$  containing
$Y$. As $V'$ is weakly special, $W$ is irreducible algebraic. By maximality of $Y$ we have 
$Y=W$. Therefore $\pi(Y)=V'$ is weakly special.

\end{proof}

\subsection{Application of Ratner's theory.}

Let $(G,X)$ be a Shimura datum and $X^+$ a connected component of $X$.
We assume that $G$ is semi-simple of adjoint type,
which we do.

We now consider a totally geodesic (weakly special) subvariety $Z$ of $X^+$.
Recall that there exists a semi-simple  subgroup $F(\RR)^+$ of $G$
without almost simple compact factors and a point $x$ such that
$x$ factors through $F Z_G(F)^0$.

Let $\alpha$ be the natural map $G(\RR)^+ \lto \Gamma \backslash G(\RR)^+$
and $\pi_x$ be the map $\Gamma \backslash G(\RR)^+ \lto \Gamma \backslash X^+$
sending $x$ to $g x$.
Recall that $\pi \colon X^+ \lto \Gamma \backslash X^+$ is the uniformisation map.
We have
$$
\overline{\pi(Z)} = \overline{\pi_x \circ \alpha (F(\RR)^+)}
$$

We let $H$ be the Mumford-Tate group of $F(\RR)^+$. Recall that it is defined to be the smallest connected 
subgroup of $G$ (hence defined over $\QQ$) whose extension to $\RR$ contains $F(\RR)^+$.


By \cite{PlaRa},  Prop 7.6, the group $F(\RR)^+$ is generated by its one-parameter 
unipotent subgroups.

By Theorem \ref{Ratnerthm}, we conclude the following:

\begin{prop} \label{Ratner}
The closure of $\alpha(F(\RR)^+)$ in $\Gamma \backslash G(\RR)^+$
is $\Gamma \cap H(\RR)^+ \backslash H(\RR)^+$.
\end{prop}

\subsection{Closure in $S$.}

From the fact that the map $\pi_x$ is proper and Proposition \ref{Ratner}, we immediately deduce the following 

\begin{teo} \label{main_th}
The closure of $\pi(Z)$ in $S$ is  $V$, the image of $H(\RR)^+ \cdot x$
i.e. it is a real weakly special subvariety.
\end{teo}

In this section we examine cases where we can actually make a stronger conclusion, namely:
\begin{enumerate}
\item The variety $V$ from Theorem \ref{main_th}
is locally symmetric and hence real totally geodesic.
\item It has a Hermitian structure i.e. is a weakly special subvariety.
\end{enumerate}

\begin{teo} \label{main_th2}
Assume $Z_G(F)$ is compact. Then $V$ is a locally symmetric variety.
\end{teo}

\begin{proof}
It is enough to show that $H(\RR)^+ \cap K_x$ is a maximal compact subgroup of $H(\RR)^+$.

Notice that since $Z_G(F)$ fixes $x$, we have 
$$
Z_G(F)\subset K_x
$$

We follow Section 3.2 of \cite{Ul}. 

Since $K_x$ is a maximal compact subgroup of $G(\RR)^+$ such that
$F(\RR)^+\cap K_x$ is a maximal compact subgroup of $F(\RR)^+$,
we have two Cartan decompositions:

$$
G(\RR)^+ = P_x K_x \text{ and } F = (P_x \cap F) \cdot (K_x \cap F)
$$

for a suitable parabolic subgroup $P_x$ of $G(\RR)^+$.

We now apply Proposition 3.10 of \cite{Ul} in out situation.
We have a connected semi-simple group $H$  such that $F \subset H_{\RR}$.
According to Proposition 3.10 of \cite{Ul}, there exists a Cartan decomposition
$$
H(\RR) = (P_x \cap H(\RR)) \cdot (K_x \cap H(\RR))
$$

This, in particular implies that $K_x \cap H(\RR)$ is a maximal compact subgroup of $H(\RR)^+$
as required.
\end{proof}

\begin{teo}\label{main_th3}
Assume that $Z_G(F)$ is trivial. Then $V$ is a weakly special subvariety.
\end{teo}
\begin{proof}
In this case, $x$ factors through $F$ and therefore through $H_{\RR}$.
Let $X_H$ be the $H(\RR)$-orbit of $x$.
By lemma 3.3 of \cite{Ul}, $(H,X_H)$ is a Shimura subdatum of $(G,X)$ and therefore
$V$ is a weakly special subvariety. 
\end{proof}

\begin{exe}
We give examples where $Z_G(F)$ is neither trivial nor compact, but the
closure of $\pi(Z)$ is nevertheless hermitian.

Let $G$ be an almost simple group over $\QQ$. 
A typical example is $G = Res_{K/\QQ} SL_{2,K}$ where $K$ is a totally real field
of degree $n$. 
Let $F$ be an $\RR$-simple factor of $G_{\RR}$. In the above case $F$ could be
for example $SL_{2}(\RR)$ embedded as $SL_{2}(\RR) \times \{1\} \times \cdots \times \{1\}$.
Then the centraliser of $F$ is not compact. However, by Lemma 2.4 of \cite{Ul}, 
the Mumford-Tate group of $F$ is $G$ and for any point $x$ of $X^+$, the image of
$F\cdot x$ in $S$ is $G$.
\end{exe}

\begin{exe}[Products of two modular curves]
Consider $G = \SL_2 \times \SL_2$, $X^+ = \HH \times \HH$ and
$$Z = \{(\tau, g \tau), \tau \in \HH \}.$$
Let $\Gamma = \SL_2(\ZZ) \times \SL_2(\ZZ)$ and
$\pi \colon \HH \times \HH \lto \Gamma \backslash X^+$.

Then, if  $g \in G(\QQ)$, then the closure of $\pi(Z)$ is a special
subvariety. It is the modular curve $Y_0(n)$ for some $n$.

If $g \notin G(\QQ)$, then $\pi(Z)$ is dense in $\Gamma \backslash X^+$. 
In this case the group $F(\RR)^+$ is $(h, ghg^{-1}) \subset \SL_2(\RR)\times \SL_2(\RR)$.
\end{exe}

\begin{exe}[Rank one groups]
Here is another quite general example where 
$Z_G(F)$ is trivial and hence the closure of the image of $F(\RR)^+\cdot x$ 
is a weakly special subvariety.

Suppose that the groups $G$ is $U(n,1)$. In this case $X^+$ is
an open ball in $\CC^n$. The real rank of $G$ is one.
Let $F$ be the subgroup $U(m,1)$ of $U(n,1)$ (with $m \leq n$).
Then the centraliser $Z_G(F)$ is trivial.
Indeed, as the split torus is already contained in $F$,
the centraliser must be compact.
\end{exe}







\vspace{0.5cm}
\small{

\noindent
Andrei Yafaev\\ University College London, Department of Mathematics.\\
Gower street, WC1E 6BT, London, UK

\noindent
email : \texttt{yafaev@math.ucl.ac.uk}}

\vspace{0.5cm}
\small{

\noindent
Emmanuel Ullmo\\ Universit\'e de Paris-Sud and IHES\\
D\'epartement de Math\'ematiques, Universit\'e Paris-Sud, B\^at. 425, F-91405 Orsay Cedex ,France

\noindent
email : \texttt{Emmanuel.Ullmo@math.u-psud.fr }}


\begin{thebibliography}{99}



\bibitem{Bo} A. Borel {\it Introduction aux groupes arithm\'etiques. }
Hermann, Paris (1969).


\bibitem{CU} L. Clozel, E. Ullmo {\it Equidistribution de sous-vari\'et\'es sp\'eciales.}  Annals of Maths., 161 (2005),1571-1588.

\bibitem{CU1}L.Clozel, E. Ullmo, {\it Equidistribution de mesures algébriques.} Compositio  Mathematica 141 (2005), 1255-1309


\bibitem{CU2} L. Clozel, E.Ullmo
{\it Equidistribution ad\`elique des tores et \'equidistribution des points CM.}  Doc. Math. 2006, Extra Vol., 233-260.


\bibitem{De} P. Deligne
{\it Vari\'et\'es de Shimura: interpr\'etation modulaire et
techniques de construction de mod\`eles canoniques},  dans {\it
Automorphic Forms, Representations, and $L$-functions} part. {\bf
2}; Editeurs: A. Borel et W Casselman; Proc. of Symp. in Pure
Math.  {\bf 33}, American Mathematical Society, (1979),  p.
247--290.





 
 
 

\bibitem{EMS} A. Eskin, S. Mozes, N. Shah, {\it Non divergence of translates of certain algebraic measures. }
GAFA vol {\bf 7} (1997) 48--80.





\bibitem{KUY} B. Klingler,E. Ullmo,  A. Yafaev {\it The Hyperbolic Ax-Lindemann-Weierstrass conjecture.} Publications Math\'ematiques de l'IHES, to appear.

\bibitem{KY} B. Klingler, A. Yafaev {\it The Andr\'e-Oort conjecture.} 
Annals of Mathematics (2), 180, 2014, n.3, 867-825. 


\bibitem{Mo} N. Mok,  {\it Metric Rigidity Theorems on Hermitian Locally Symmetric Manifolds}
Series in Pure Math. {\bf 6}. World Scientific  (1989).

\bibitem{MoMo} B. Moonen {\it Linearity properties of Shimura varieties. I.} J. Algebraic Geom. 7 (1998), 539-567.





\bibitem{PlaRa} V. Platonov, A. Rapinchuk {\it Algebraic groups and number theory.}
Academic Press, Pure and Applied Mathematics, 139 



\bibitem{PZ} J. Pila, U. Zannier 
{\it Rational points in periodic analytic sets and the Manin-Mumford conjecture}. Atti Accad. Naz. Lincei Cl. Sci. Fis. Mat. Natur. Rend. 
Lincei (9) Mat. Appl. {\bf 19} (2008), no. 2, 149-162.


\bibitem{Rat} M. Ratner {\it Interaction between Ergodic theory, Lie groups and Number theory.}
Proceedings of ICM 1994 (Zurich), Vol 1, Birkhouser (1995), p. 157-182.





\bibitem{PlaRa} V. Platonov, A. Rapinchuk {\it Algebraic groups and number theory.} Academic Press, 1994



\bibitem{UY_AV} E. Ullmo, A. Yafaev {\it Algebraic flows on abelian varieties.} 
to appear in J. reine und angewandte Mathematik.

\bibitem{Ul} E. Ullmo {\it Equidistribution des sous-vari\'et\'es sp\'eciales II.} 
J. reine angew. Math.  (2006), 1-24.
 
\bibitem{Ul2} E. Ullmo {\it Applications du theor\`eme d'Ax-Lindemann hyperbolique.} Compositio Mathematica, 150, n.2, 2014, 175-190.



\bibitem{UY} E. Ullmo, A. Yafaev {\it A characterisation of special subvarieties.} 
Mathematika {\bf 57} (2011), no. 2, 263-273.


\bibitem{UY1} E. Ullmo, A. Yafaev {\it The hyperbolic Ax-Lindemann conjecture in the co-compact case.}  Duke Math Journal, 163, 2014, n.2, 433-463.

\end{thebibliography}
\end{document}